\documentclass[11pt,reqno]{amsart}

\usepackage{amsmath,amssymb, graphicx,latexsym,amsthm,mathrsfs,amsthm}
\usepackage{graphicx,color,cite}
\usepackage{caption, subcaption}
\newtheorem{theorem}{Theorem}
\newtheorem{lemma}{Lemma}
\newtheorem{proposition}{Proposition}
\newtheorem{assumption}{Assumption}

\newtheorem{remark}{Remark}
\theoremstyle{remark}

\begin{document}

\title[An inverse source problem]{Lipschitz stability for an inverse source
scattering problem at a fixed frequency}

\author[P. Li]{Peijun Li}
\address{Department of Mathematics, Purdue University, West Lafayette, Indiana
47907, USA}
\email{lipeijun@math.purdue.edu}

\author[J. Zhai]{Jian Zhai}
\address{Institute for Advanced Study, The Hong Kong University of Science and
Technology, Kowloon, Hong Kong, China}
\email{iasjzhai@ust.hk}
  
\author[Y. Zhao]{Yue Zhao}
\address{School of Mathematics and Statistics, Central China Normal University, Wuhan 430079, China}
\email{zhaoyueccnu@163.com}

\thanks{The research of PL is supported in part by the NSF grant
DMS-1912704.}



\begin{abstract}
This paper is concerned with an inverse source problem for the three-dimensional
Helmholtz equation by a single boundary measurement at a fixed frequency. We
show the Lipschitz stability under the assumption that the source function is
piecewise constant on a domain which is made of a union of
disjoint convex polyhedral subdomains.
\end{abstract}

\maketitle

\section{Introduction}

The inverse source scattering problems arise in diverse scientific and
industrial areas such as antenna design and synthesis, medical imaging
\cite{FKM-IP}. In
general there is no uniqueness for the inverse source scattering problems with
the boundary data at a fixed frequency \cite{bc-77}. This is clear since a single near-field or
far-field measurement gives a function of $n-1$ independent variables in an 
$n$-dimensional space, while the source function has $n$ independent variables. An effective approach to
overcome the non-uniqueness issue is the use of multi-frequency data. More
interestingly, the use of multi-frequency data may enhance the stability of the problems \cite{actv, blrx,
BLLT, blt, blz, LY-17, LZZ, cheng2016increasing}.

Nevertheless, with single-frequency data, it is proved in \cite{ks-cpam, s} that the support of the source can still be determined in
certain cases.
In \cite{ikehata}, it was shown that the convex
hull of a polygonal source can be determined from a single measurement. For
sources with a convex polygonal support, it has been proved that the support and the values of the source function at
corner points can be uniquely determined by a single measurement in
homogeneous \cite{blasten} and inhomogeneous media \cite{hl}.  In \cite{bps},
the
authors addressed the absence of real non-scattering energies by examining the
phenomenon that corners always scatter. Related studies can be found in
\cite{FI} and \cite{eh} on the uniqueness of the shape identification by using a
single measurement in the inverse conductivity and medium scattering problems,
respectively. We refer to \cite{abf, bt} for the uniqueness and numerical
results for recovering point and dipole sources.

Consider the three-dimensional Helmholtz equation
\begin{align}\label{eqn}
\Delta u(x) + \kappa^2 u(x) = f(x), \quad x\in\mathbb{R}^3,
\end{align}
where $\kappa>0$ is the wavenumber, $u$ denotes the wave field, and the
source function $f\in L^\infty(\mathbb R^3)$ represents the electric current
density and is assumed to have a compact support contained in a bounded
domain $\Omega\subset\mathbb R^3$ with a connected complement $\mathbb
R^3\backslash \overline{\Omega}.$ Furthermore, we assume that
$\overline{\Omega}\subset B_R:=\{x\in\mathbb R^3: |x|<R\}$, where
$R>0$ is a constant. The wave field $u$ is required to satisfy the
Sommerfeld radiation condition
\begin{align}\label{src}
\lim_{r\to\infty}r (\partial_r u-\mathrm{i}\kappa u)=0,\quad r=|x|
\end{align}
uniformly in all directions $\hat{x} = x/|x|$.

Given the source $f$, the direct scattering problem is to determine the wave
field $u$ which satisfies \eqref{eqn}--\eqref{src}. It is known that the direct 
scattering problem  has a unique solution $u\in H^2(B_R)$ for an arbitrary wavenumber
$\kappa>0$  and the solution $u$ satisfies the following estimate (cf.
\cite{CK}): 
\begin{equation}\label{est_u}
\|u\|_{H^2(B_R)}\leq C\|f\|_{L^\infty(\Omega)},
\end{equation}
where $C$ is a positive constant. This paper is concerned with the inverse
source scattering problem, which is to determine $f$ from the boundary
measurement of $u$ on $\partial B_R=\{x\in\mathbb R^3: |x|=R\}$ at a fixed
wavenumber $\kappa$. 

In this work, we consider the case where the source
$f$ is a piecewise constant function. More precisely, we assume 
\begin{align}\label{formf}
f(x) = \sum_{j=1}^{N} c_j\chi_{D_j}(x),
\end{align}
where $D_j,\, j=1, \cdots, N$ are known disjoint convex polyhedral domains and $c_j,\,
j=1, \cdots, N$ are unknown constants. The goal is to establish the Lipschitz
stability of determining the constants $c_j,\, j=1, \cdots, N$ from the 
measurement of $u$ on $\partial B_R$ at a fixed wavenumber $\kappa$. 
It is known that there exist certain sources that produce no
measurable signals, and those sources are called non-radiating sources
\cite{bc-77}.  However, since the support of the source function
\eqref{formf} has corners, it is a radiating source (cf.
\cite{blasten}). This makes the recovery of $f$ possible. We refer to
\cite{actv, AS} for the characterization of radiating and non-radiating sources
for the Helmholtz equation and Maxwell equations.

Our study is motivated by the idea introduced by Alessandrini and Vessella in
\cite{av}, where the electrical impedance tomography problem was studied. This
approach was further developed to study various inverse coefficient problems
(cf., \cite{adgs,bf, BFMRV, bhq, bhz,bfv}). In this paper, we use similar ideas
to solve our inverse source problem. In \cite{blasten, ikehata}, the inverse
source problems are studied by using complex geometric optics (CGO)
solutions, which are also typical mathematical tools for the inverse coefficient
problems \cite{calderon, su}.

We construct singular solutions and utilize their ``blow-up"
behaviors near the corners of subdomains $D_j$, $j=1,2,\cdots, N$. The quantitative
estimate of unique continuation of the
solution for the Helmholtz equation, which is derived from a three spheres
inequality, plays an essential role in the procedure. We derive
a logarithmic-type stability for recovering $c_1,c_2,\cdots, c_N$, and then
uniqueness follows immediately. Since we are recovering a finite number of
unknowns, the Lipschitz-type stability estimate is obtained.  Comparing with
the uniqueness results in \cite{blasten, hl},
we provide the uniqueness for a different class of source functions and achieve
the optimal stability estimate.
We also want to point out that recently there are numerous results
of establishing Lipschitz stability for some inverse problems using finite
measurements (cf. \cite{AS, lt, rs, harrach} for the Calder\'{o}n problem and
\cite{BH-IP} for inverse scattering problems).

The paper is organized as follows. In Section \ref{mr}, we summarize the main
results. Section \ref{pf} is devoted to the proof of the main result. The paper
is concluded with some general remarks and directions for future work in Section
\ref{conclusion}.

\section{Main result}\label{mr}

In this section, we make some extra assumptions on the source function and
state the main result of this work. 

\subsection{Geometry setup}

Let the piecewise constant source function be given as 
\begin{align*}
f(x) = \sum_{j=1}^N c_j \chi_{D_j}(x), \quad \overline{\Omega} =
\cup_{j=1}^N\overline{D}_j,
\end{align*}
where $c_j\in\mathbb C$ are constants, and $D_j$ are mutually disjoint bounded
open subsets in $\mathbb R^3$. Assume that
$\mathrm{dist}(\Omega,\mathbb{R}^3\setminus B_R)\geq r_0$ for some constant
$r_0>0$. Moreover, we consider the geometric setup of the
domains $D_j$ that can be described as the polyhedral cell geometry as follows
(cf. \cite{BH-IP}).

\begin{assumption}\label{gs}
We assume that 
\begin{enumerate}
\item the subdomains $D_j\subset\mathbb R^3, 1\leq j \leq N$ are
convex polyhedrons;


\item for each $k=0,\cdots, N-1$, $\cup_{j=k+1}^N\overline{D_j}$ is simply
connected, and there exists a constant $r_0$ such that $\{x\in\mathbb{R}^3\vert
\mathrm{dist}(x,\cup_{j=k+1}^N\overline{D_j})>2r_0\}$ is connected;

\item each $D_j$ has a vertex, denoted by $P^{(j)}$, such that $B_{3r_0}(P^{(j)})\cap
D_k=\emptyset$ for any $k>j$.

\end{enumerate}
\end{assumption}

An example domain in $\mathbb{R}^2$ satisfying the above assumptions is
illustrated in Figure \ref{domain1}.\\

\begin{figure}[htbp]
\centering
\includegraphics[width=0.4\textwidth]{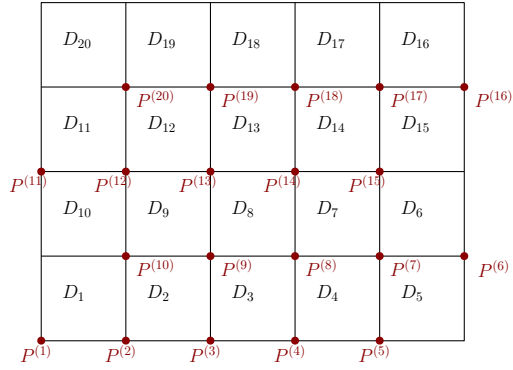}
\caption{An example of the domain.}
\label{domain1}
\end{figure}

Let $(x_1,x_2,x_3)$ be the Cartesian coordinate in $\mathbb{R}^3$, and introduce the spherical coordinates
\[
x_1=\rho\sin\theta\cos\varphi,\quad x_2=\rho\sin\theta\sin\varphi,\quad x_3=\rho\cos\theta.
\]
Assume $\alpha=\alpha(\varphi)$ is a continuous function on $[0,2\pi]$, such that $\alpha(\varphi)\in (0,\frac{\pi}{2})$ for any $\varphi\in [0,2\pi]$. We let 
\[
\mathcal{C}(r_0, \alpha): = \{(\rho, \theta,\varphi): 0\leq\rho\leq r_0,\,
0\leq \theta\leq \alpha(\varphi),0\leq\varphi\leq 2\pi\}
\]
 denote the
cone with radius $r_0$ and vertical angle $\alpha$. The vertex of the cone is the origin and the axis is the $x_3$-axis. The cone $\mathcal{C}(r_0,\alpha)$ is depicted in Figure \ref{cone_eg}.
\begin{figure}[htbp]
\centering
\includegraphics[width=0.4\textwidth]{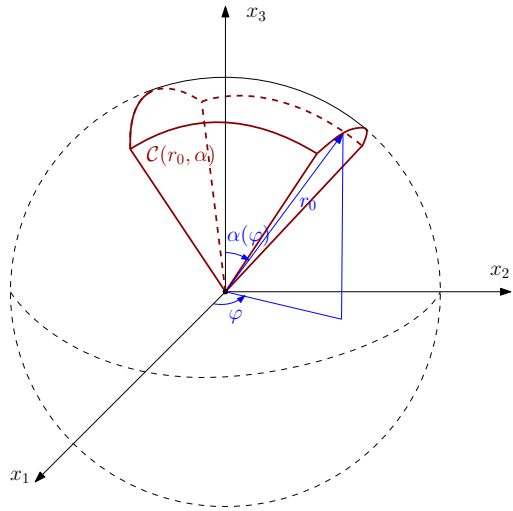}
\caption{Illustration of $\mathcal{C}(r_0,\alpha)$.}
\label{cone_eg}
\end{figure}

\begin{assumption}\label{convexp}
Let $\alpha_1,\alpha_2$ be two constants satisfying
$0<\alpha_1<\alpha_2<\frac{\pi}{2}$. For each $D_j$, $j=1,2\cdots, N$, let $P^{(j)}_\ell$
be a vertex. Assume that, after a rigid transform, $P^{(j)}_\ell=(0,0,0)$, and $B_{r_0}\cap
D_j=\mathcal{C}(r_0,\alpha^{(j)}_\ell)$ with
$\alpha_1<\alpha^{(j)}_\ell(\varphi)<\alpha_2$ for any $\varphi\in[0,2\pi]$.
\end{assumption}

In addition, we also make the following assumption on the source function. 

\begin{assumption}\label{assumption}
The source function $f$ has the compact support $\overline{\Omega}$ with
$|\Omega|\leq A$ and satisfies $\|f\|_{L^\infty(\Omega)}\leq E$, where $A$ and
$E$ are positive constants.
\end{assumption}

\subsection{Statement of the main result}

 Denote
\begin{align*}
\epsilon := \|u\|_{H^1(\partial B_R)}.
\end{align*}

The following theorem is the main result of this paper.

\begin{theorem}\label{main_theorem}
Let $f$ satisfy Assumptions \ref{gs}--\ref{assumption} and the subdomains $D_j,
j=1, \dots, N$ are given. Then the following estimate holds:
\begin{align}\label{lip_stability}
\|f\|_{L^\infty(\Omega)}\lesssim \epsilon.
\end{align}
\end{theorem}

Hereafter, the notation $a\lesssim b$ stands for $a\leq Cb$, where $C>0$ is a
positive constant which depends on the following parameters: $\kappa, A,E,N,
r_0,R,\alpha_1,\alpha_2$. 

\begin{remark}
It is clear to note that the optimal Lipschitz stability \eqref{lip_stability}
implies the uniqueness, i.e., if $\epsilon = 0$, then $f=0$. We mention that
the Lipschitz constant in the estimate \eqref{lip_stability} grows exponentially
with respect to the number of subdomains $N$, which means that the
stability estimate deteriorates dramatically as $N$ grows. We refer to
\textnormal{\cite{rondi,bhq}} for related studies of this behavior.
The Lipschitz constant also deteriorates when the number $r_0$ decreases due to
the instability of the unique continuation principle and the use of increased
number of three spheres inequalities.
\end{remark}

\subsection{Construction of singular solutions}\label{ss}

To prove the theorem, we need to construct singular solutions to the Helmholtz
equation and use their asymptotic behaviors near the singularities.
For the inverse coefficient problems considered in \cite{adgs,av, bf, BFMRV,
bhq, bhz,bfv}, typically one may deal with a product of two singular
solutions, whose positivity can be guaranteed. For our inverse source problem,
we deal with only one singular solution, and therefore more sophisticated
analysis is needed. In particular, we need to derive a lower bound on the
integral of the singular solution over a cone, when the singular point is
outside the cone and close to the vertex. One will see that the cone has to be
strictly convex at the vertex in order to have such a bound. Since this is the
key difference from previous work on the inverse coefficient problems, we 
provide more details in this section.

Denote by $G(x) = \frac{e^{\mathrm{i}\kappa|x|}}{|x|}$ the fundamental
solution to
the three-dimensional Helmholtz equation in a homogeneous medium. By simple
calculations, we obtain for sufficiently small $|x|$ that 
\begin{align*}
\partial^3_{x_3}\frac{e^{\mathrm{i}\kappa|x|}}{|x|}&\sim
\left(\frac{3x_3}{|x|^5}+\frac{6x_3}{|x|^5}-\frac{15x_3^3}{|x|^7}\right)e^{
\mathrm{i}\kappa|x|}+\mathcal{O}(|x|^{-3})\\
&=\frac{x_3(9x_1^2+9x_2^2-6x_3^2)}{|x|^7}e^{
\mathrm{i}\kappa|x|}+\mathcal{O}(|x|^{-3}).
\end{align*}
We will use the following singular solution

\begin{equation}\label{Phi_asymp}
\begin{split}
\Phi(x)=-\mathrm{Im}\Big(\partial^3_{x_3}\frac{e^{\mathrm{i}\kappa|x|}}{|x|}\Big)=&\frac
{x_3(-9x_1^2-9x_2^2+6x_3^2)}{|x|^7}\cos(\kappa|x|)+\mathcal{O}(|x|^{-3})\\
=&\frac
{x_3(-9x_1^2-9x_2^2+6x_3^2)}{|x|^7}+\mathcal{O}(|x|^{-3}),
\end{split}
\end{equation}

which has a singularity at $x=0$.

Consider a cone $\mathcal{C}(r_0, \alpha)= \{(\rho, \theta,\varphi): 0\leq\rho\leq r_0,
0\leq \theta\leq \alpha(\varphi),0\leq\varphi\leq 2\pi\}$, with $\alpha_1<\alpha(\varphi)<\alpha_2$ for any $\varphi\in[0,2\pi]$. We assume that $0<\alpha_1<\alpha_2<\frac{\pi}{2}$, and then the cone $\mathcal{C}(r_0, \alpha)$ is convex near the vertex. For our purpose, one can think this vertex as a corner of some $D_j$.
Denote
\[
\mathcal{C}'(r,r_0,\alpha):=B_{r_0}(0)\cap\{(0,0,r)+\mathcal{C}(r_0,\alpha)\}.
\]
See Figure \ref{cone}(A) for an illustration.
 Substitute
\[
x_1=r\widetilde{\rho}\sin\widetilde{\theta}\cos\varphi,\quad x_2=r\widetilde{\rho}\sin\widetilde{\theta}\sin\varphi,\quad x_3=r\widetilde{\rho}\cos\widetilde{\theta}.
\]
Then $\mathcal{C}'(r,r_0,\alpha)$ can be expressed as
\[
\mathcal{C}'(r,r_0,\alpha)=\{(\widetilde{\rho},\widetilde{\theta},\varphi):\, 1\leq\widetilde{\rho}\leq\frac{r_0}{r},0\leq \theta\leq \widetilde{\alpha}(\varphi,\widetilde{\rho}),0\leq\varphi\leq 2\pi\},
\]
for some $\widetilde{\alpha}$ satisfying $\widetilde{\alpha}(\varphi,\widetilde{\rho})<\alpha_2$ and $\widetilde{\alpha}(\varphi,1)=0$.

\begin{figure}[t!]
    \begin{subfigure}{0.35\textwidth}
      \includegraphics[width=\textwidth]{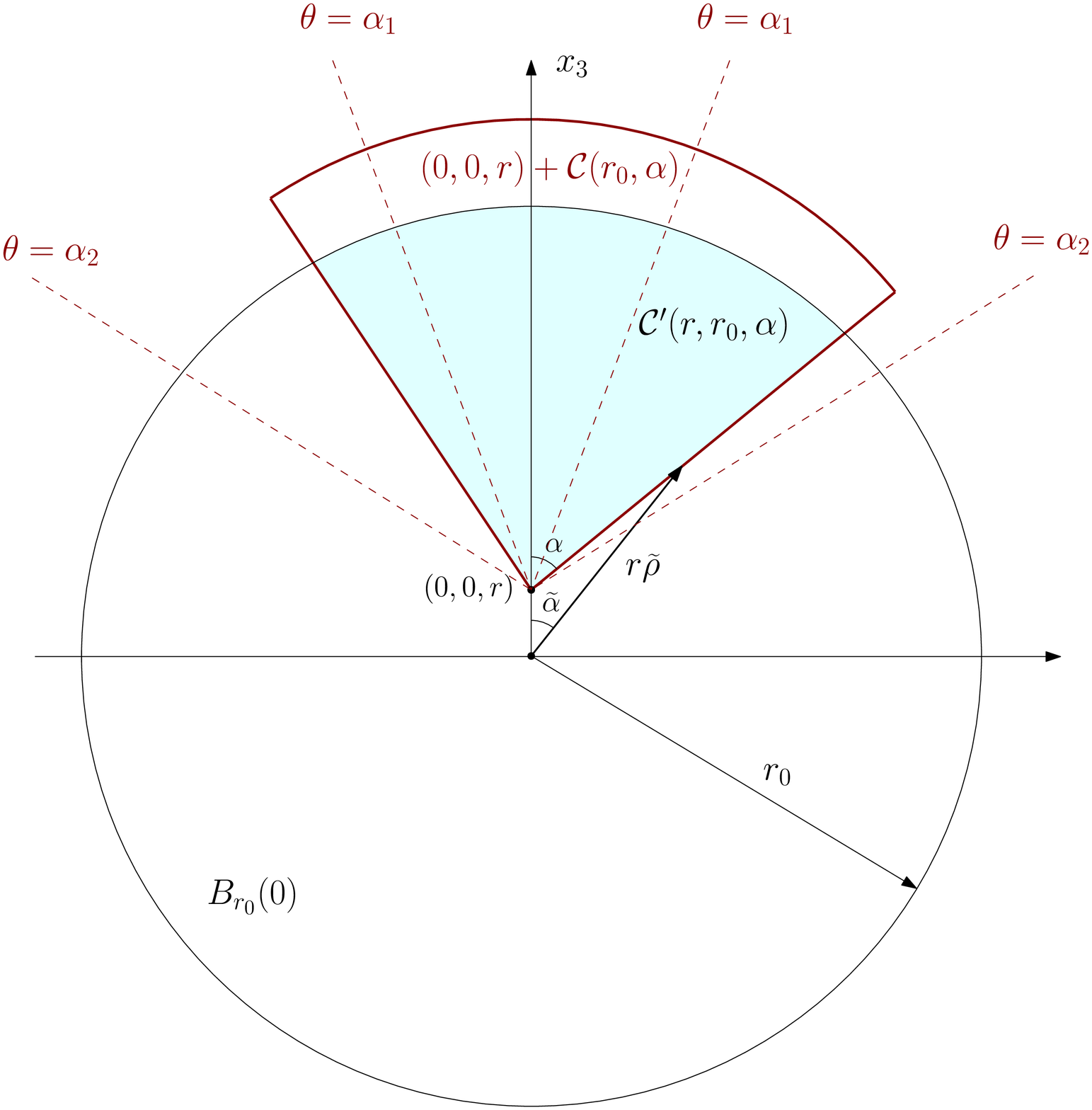}
      \caption{Illustration of $\mathcal{C}(r_0,\alpha)$ and $\mathcal{C}'(r,r_0,\alpha)$.}
    \end{subfigure}
    \hspace{1in}
    \begin{subfigure}{0.35\textwidth}
      \includegraphics[width=\textwidth]{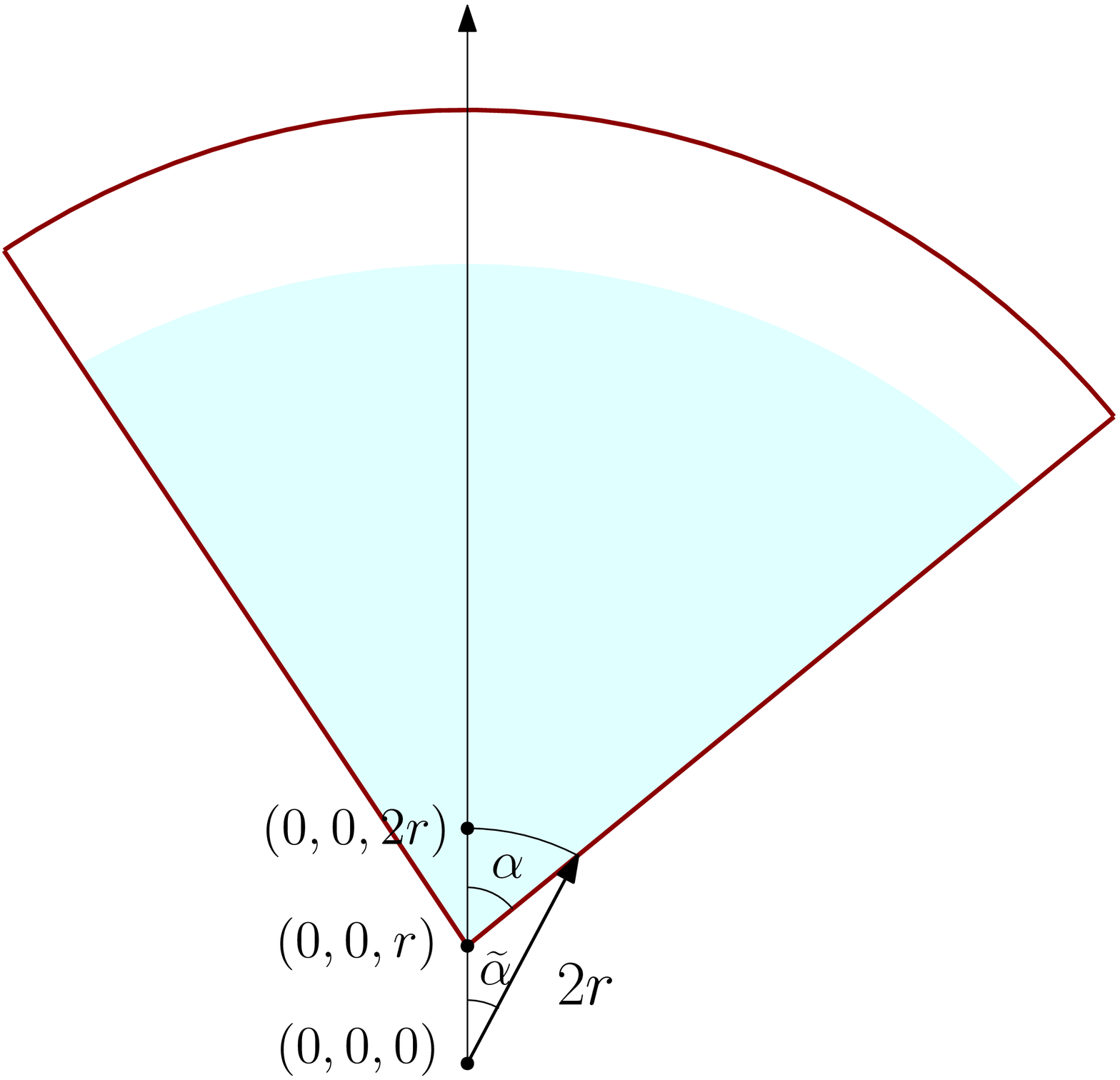}
      \caption{One can see clearly that $\frac{\alpha(\varphi)}{2}<\widetilde{\alpha}(\varphi,\widetilde{\rho})<\alpha(\varphi)$ for $\widetilde{\rho}>2$.}
    \end{subfigure}
    \caption{Illustrations of the domains defined.}\label{cone}
  \end{figure}

 By taking
the integral of $\frac{x_3(-9x_1^2-9x_2^2+6x_3^2)}{|x|^7}$ in
$\mathcal{C}'(r,r_0,\alpha)$ for small $r$,
 we get 
\begin{align}\label{est_C}
&\int_{\mathcal{C}'(r,r_0,\alpha)}\frac{x_3(-9x_1^2-9x_2^2+6x_3^2)}{|x|^7}
\mathrm{d}x\notag\\
&=r^{-1}\int_{1}^{r_0/r}{\rm d}\widetilde{\rho}\int_0^{2\pi}{\rm
d}\varphi\int_0^{\widetilde{\alpha}(\varphi,\widetilde{\rho})}{\rm
d}\widetilde{\theta}\left[\widetilde{\rho}^{-2}
\sin\widetilde{\theta}\cos\widetilde{\theta}(-9\sin^2\widetilde{\theta}+6\cos^2\widetilde{\theta})\right].
\end{align}
Next we will bound the above integral from below for $r> 0$ small.
For any $\widetilde{\rho}\in[1,\frac{r_0}{r}]$ and $\varphi\in[0,2\pi]$, since
$\widetilde{\alpha}(\varphi,\widetilde{\rho})\in [0,\frac{\pi}{2}]$, we have
\begin{align*}
&\int_0^{\widetilde{\alpha}(\varphi,\widetilde{\rho})}\left[
\sin\widetilde{\theta}\cos\widetilde{\theta}(-9\sin^2\widetilde{\theta}+6\cos^2\widetilde{\theta})\right]{\rm
d}\widetilde{\theta}\\
&=3(\cos^3\widetilde{\theta}-\cos^5\widetilde{\theta})\Big\vert_{0}^{\widetilde
{\alpha}(\varphi,\widetilde{\rho})}\\
&=3\cos^3\widetilde{\alpha}(\varphi,\widetilde{\rho})-3\cos^5\widetilde{\alpha}
(\varphi,\widetilde{\rho})\\
&\geq 0.
\end{align*}
 By elementary geometry, we have for any $\widetilde{\rho}>2$
that  
 \[
 0<\frac{\alpha_1}{2}<\frac{\alpha(\varphi)}{2}<\widetilde{\alpha}(\varphi,\widetilde{\rho})<\alpha(\varphi)<\alpha_2<\frac{\pi}{2},
 \]
 which is illustrated in Figure \ref{cone}(B), and then
\begin{align*}
&\int_0^{\widetilde{\alpha}(\varphi,\widetilde{\rho})}\left[
\sin\widetilde{\theta}\cos\widetilde{\theta}(-9\sin^2\widetilde{\theta}+6\cos^2\widetilde{\theta})\right]{\rm
d}\widetilde{\theta}\\
&=3\cos^3\widetilde{\alpha}(\varphi,\widetilde{\rho})-3\cos^5\widetilde{\alpha}
(\varphi,\widetilde{\rho})\\
&\geq 3\min\left\{\cos^3\frac{\alpha_1}{2}-\cos^5\frac{\alpha_1}{2},
\cos^3\alpha_2-\cos^5\alpha_2\right\}\\
&> 0,
\end{align*}
for $\widetilde{\rho}>2$.
Thus we obtain
\begin{align*}
&\int_{1}^{+\infty}{\rm d}\widetilde{\rho}\int_0^{2\pi}{\rm
d}\varphi\int_0^{\widetilde{\alpha}(\varphi,\widetilde{\rho})}{\rm
d}\widetilde{\theta}\left[\widetilde{\rho}^{-2}
\sin\widetilde{\theta}\cos\widetilde{\theta}(-9\sin^2\widetilde{\theta}+6\cos^2\widetilde{\theta})\right]\\
&\geq \int_{2}^{+\infty}{\rm d}\widetilde{\rho}\int_0^{2\pi}{\rm
d}\varphi\int_0^{\widetilde{\alpha}(\varphi,\widetilde{\rho})}{\rm
d}\widetilde{\theta}\left[\widetilde{\rho}^{-2}
\sin\widetilde{\theta}\cos\widetilde{\theta}(-9\sin^2\widetilde{\theta}+6\cos^2\widetilde{\theta})\right]\\
&\geq C_0,
\end{align*}
where the constant $C_0>0$ depends on $\alpha_1,\alpha_2$. We also have 
\[
\left\vert\int_{r_0/r}^{+\infty}{\rm d}\widetilde{\rho}\int_0^{2\pi}{\rm
d}\varphi\int_0^{\widetilde{\alpha}(\varphi,\widetilde{\rho})}{\rm
d}\widetilde{\theta}\left[\widetilde{\rho}^{-2}
\sin\widetilde{\theta}\cos\widetilde{\theta}(-9\sin^2\widetilde{\theta}+6\cos^2\widetilde{\theta})\right]\right\vert\leq C'\left\vert\int_{r_0/r}^{\infty}\widetilde{\rho}^{-2}\mathrm{d}\widetilde{\rho}\right\vert\leq C'r,
\]
where $C'$ is a positive constant. Therefore
\begin{align*}
&\int_{1}^{r_0/r}{\rm d}\widetilde{\rho}\int_0^{2\pi}{\rm
d}\varphi\int_0^{\widetilde{\alpha}(\varphi,\widetilde{\rho})}{\rm
d}\widetilde{\theta}\left[\widetilde{\rho}^{-2}
\sin\widetilde{\theta}\cos\widetilde{\theta}(-9\sin^2\widetilde{\theta}+6\cos^2\widetilde{\theta})\right]\\
& \geq \int_{1}^{+\infty}{\rm d}\widetilde{\rho}\int_0^{2\pi}{\rm
d}\varphi\int_0^{\widetilde{\alpha}(\varphi,\widetilde{\rho})}{\rm
d}\widetilde{\theta}\left[\widetilde{\rho}^{-2}
\sin\widetilde{\theta}\cos\widetilde{\theta}(-9\sin^2\widetilde{\theta}+6\cos^2\widetilde{\theta})\right]\\
&\quad -\left\vert\int_{r_0/r}^{+\infty}{\rm d}\widetilde{\rho}\int_0^{2\pi}{\rm
d}\varphi\int_0^{\widetilde{\alpha}(\varphi,\widetilde{\rho})}{\rm
d}\widetilde{\theta}\left[\widetilde{\rho}^{-2}
\sin\widetilde{\theta}\cos\widetilde{\theta}(-9\sin^2\widetilde{\theta}+6\cos^2\widetilde{\theta})\right]\right\vert\\
&\geq C_0-C'r.
\end{align*}
Using the above estimate and \eqref{est_C}, we have
\begin{equation}\label{bp}
\int_{\mathcal{C}'(r,r_0,\alpha)}\Phi(x)\mathrm{d}x\geq C_0r^{-1}-C_1|\log r|,
\end{equation}
where $C_0>0,\, C_1$ depend on $\alpha_1,\alpha_2,r_0,\kappa$, and we have
used the asymptotics of $\Phi$ given in \eqref{Phi_asymp} and the fact that
\[
\int_{\mathcal{C}'(r,r_0,\alpha)}|x|^{-3}\mathrm{d}x\leq C\int_{B_{r_0}\setminus B_r}|x|^{-3}\mathrm{d}x\leq C\int_{r}^{r_0}\rho^{-3}\rho^2\mathrm{d}\rho\leq C|\log r|.
\]

For $x\neq y$, we define
\[
G(x,y):=G(x-y)
\]
and
\begin{equation}\label{def_Phi}
\Phi(x, y) := \Phi(x-y)=-\mathrm{Im}(\partial^3_{x_3}G(x,y))=-\mathrm{Im}(\partial^3_{x_3}G(x-y)).
\end{equation}
It is easy to verify that
\[
\Phi(y, x)=\Phi(y-x)=-\mathrm{Im}(\partial^3_{y_3}G(x-y))=-\Phi(x,y)=-\Phi(x-y).
\]
For fixed $y$, it is clear to see that the function $\Phi(\cdot\,,\, y)$ is
singular at $x=y$ and satisfies the Helmholtz equation for $x\neq
y$.

\begin{remark}
The estimate \eqref{bp} with the constant $C_0>0$ is crucial for the proof of the main theorem. We can not have a positive $C_0$ near a facet point, for
which $\alpha\equiv\alpha_2 =\frac{\pi}{2}$.This is the fact that
corners always have strong scattering effects \textnormal{\cite{bps, blasten}}.
Therefore, we
will be essentially using ``corner scattering" to do the recovery. We refer to
\textnormal{\cite{BH-IP,XDL}} for similar approaches to recover piecewise constant
coefficients. We believe that one can also use ``edge scattering"
to serve our purposes.
\end{remark}

\section{Proof of the main result}\label{pf}

In this section, we show the proof of the main result which is stated in
Theorem \ref{main_theorem}. First we define a sequence of domains
which
will be used in the proof. 

Let
\[
U_0 = \Omega, \quad W_0=\emptyset, \quad U_k = \Omega\setminus\cup_{j=1}^kD_j,\quad W_k=\Omega\setminus U_k, \quad k = 1, ... ,
N.
\]
For each $k\in\{0,1, 2, ... , N-1\}$, consider the vertex $P^{(k+1)}$ of the
cell $D_{k+1}$. By choosing appropriate Cartesian coordinates
$(x_1^{(k+1)},x_3^{(k+1)},x_3^{(k+1)})$, we assume $D_{k+1}\cap
B_{r_0}(P^{(k+1}))=P^{(k+1)}+\mathcal{C}(r_0, \alpha^{(k+1)})$, with
$\alpha^{(k+1)}=\alpha^{(k+1)}(\varphi)$, $\varphi\in[0,2\pi]$, i.e., a cone
with vertex at $P^{(k+1)}$. By Assumption \ref{convexp}, we have
\[
\alpha_1<\alpha^{(k+1)}(\varphi)<\alpha_2
\]
for $\varphi\in [0,2\pi]$.

Denote $P^{(k+1)}=(p^{(k+1)}_1,p^{(k+1)}_2,p^{(k+1)}_3)$,
\[
\begin{split}
Q_{k+1}^-=\{x=(x_1^{(k+1)},x_2^{(k+1)},x_3^{(k+1)}):\,&|x_1^{(k+1)}-p^{(k+1)}_1|^2+|x_2^{(k+1)}-p^{(k+1)}_2|^2<
r_0^2,\,\\
&-2r_0<x_3^{(k+1)}-p^{(k+1)}_3<0\},
\end{split}
\]
and
\[
\mathcal{K}_k=\{x\in B_{R+r_0}:\mathrm{dist}(x,U_k)>r_0\}\cup Q_{k+1}^-.
\]
We note that $\mathcal{K}_k$ is connected under Assumption
\ref{gs}. Figure \ref{domain} shows an illustrative example of the domains
$U_k$ and
$\mathcal K_k$. 

\begin{figure}[htbp]
\centering
\includegraphics[width=0.4\textwidth]{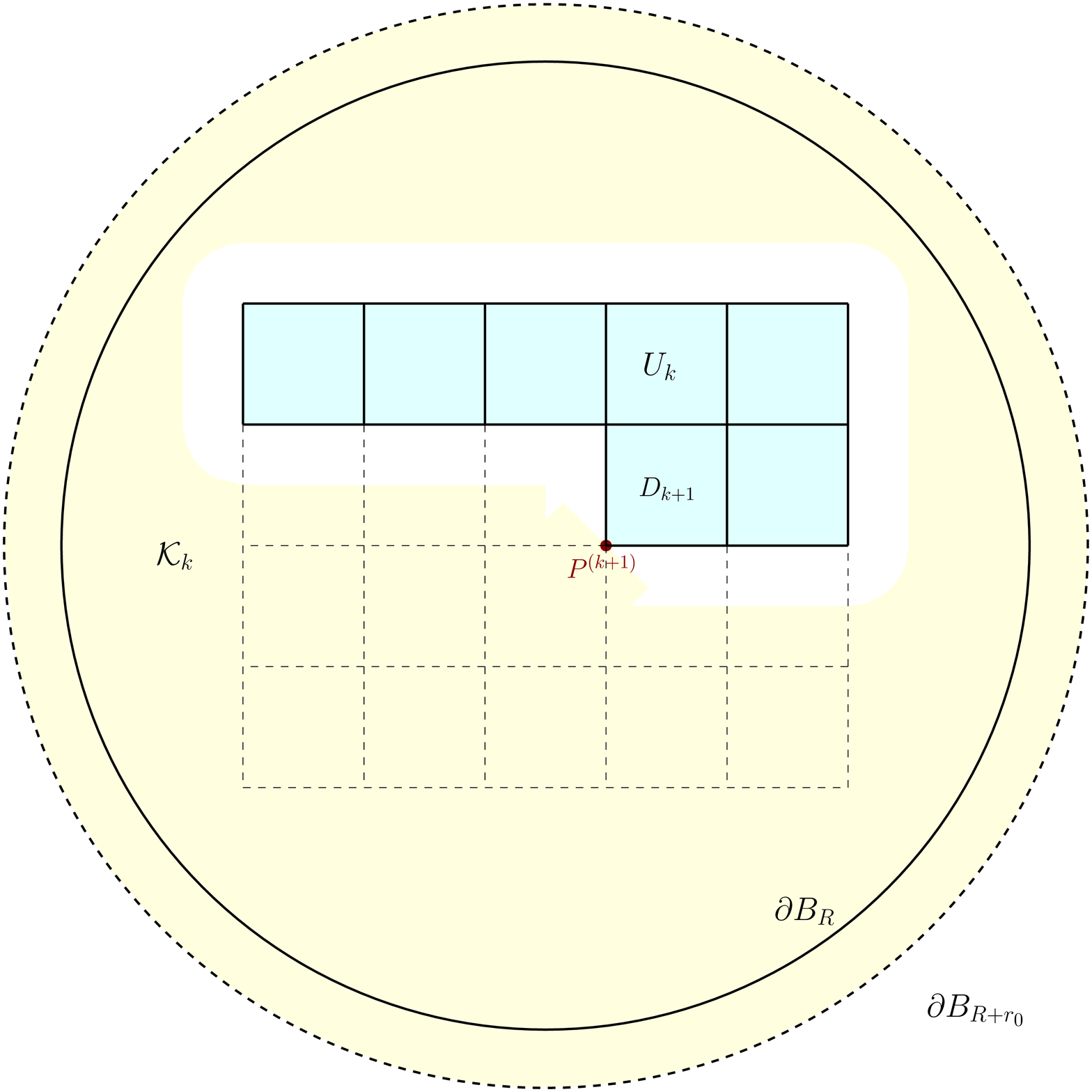}
\caption{The domains $U_k$ and $\mathcal{K}_k$ for $k=13$.}
\label{domain}
\end{figure}

\subsection{Unique continuation}

We state a quantitative estimate of unique continuation for the solution of the
Helmholtz equation. The proof is omitted since it is a minor modification of the
proof for a similar estimate in \cite[Proposition 3.9]{bhq} and
\cite[Proposition 7]{bfv}. We remark  that the proof is based on
the construction of a pathway and the repeated use of three spheres inequalities
under Assumption \ref{gs}.

\begin{proposition}\label{estimate}
Let $\mathcal{K}_k$ be defined as before and let $v\in H^1(\mathcal{K}_k)$ be a
weak solution to the Helmholtz equation
\[
\Delta v + \kappa^2 v = f \quad \text{in} \quad \mathcal{K}_k.
\]
Assume that, for given positive constants $\varepsilon_0$ and $E_1$, $v$ satisfies
\begin{align*}
\|v\|_{L^\infty(B_{R+r_0}\setminus B_{R+\frac{r_0}{2}}))} \leq \varepsilon_0
\end{align*}
and 
\begin{align*}
|v(x)|\leq E_1 |x - P^{(k+1)}|^{-1}, \quad x\in \mathcal{K}_k.
\end{align*}
Then the following inequality holds for small enough $r>0$:
\begin{align*}
|v(x_r)|\lesssim \varepsilon^{\tau_r}E_1^{1-\tau_r}r^{-(1-\tau_r)},
\end{align*}
where $x_r= P^{(k+1)} +(0,0,-r)$ and $\tau_r=\theta r^\delta$ with
$0<\theta<1$ and $\delta>0$ depending on $r_0,\kappa, N,A$.
\end{proposition}

\subsection{Proof of Theorem \ref{main_theorem}}

For some $k\in\{0,1,\cdots,N-1\}$, let
\[
\Phi_k(x,y):=-\mathrm{Im}(\partial^3_{x_3^{(k+1)}}G(x-y)).
\]
For a fixed $k$, we just denote the Cartesian coordinates
$(x_1,x_2,x_3)=(x_1^{(k+1)},x_3^{(k+1)},x_3^{(k+1)})$ for brevity. In the
following, we work exclusively under this coordinate system. Note that,
under these coordinates, formally we have
\[
\Phi_k(x,y)=\Phi(x-y).
\]
where $\Phi(\cdot,\cdot)$ is defined in \eqref{def_Phi}.

 Define
\[
S_k(y) = \int_{U_k}f(x)\Phi_k(x, y){\rm d}x.
\]

\begin{lemma}\label{S_sol}
For $y\in \mathcal{K}_k$, it holds that $(\Delta + \kappa^2)S_k(y) = 0.$
\end{lemma}

\begin{proof}
Noting that for any $x\in U_k$, $y\in \mathcal{K}_k$, we have
\begin{align*}
f(x)(\Delta_y+\kappa^2)\Phi_k(x,
y)&=-f(x)(\Delta_y+\kappa^2)\mathrm{Im}(\partial^3_{x_3}G(x-y))\\
&=-f(x)\partial^3_{x_3}\mathrm{Im}((\Delta_y+\kappa^2)G(x-y))\\
&=0,
\end{align*}
since $U_k$ and $\mathcal{K}_k$ are disconnected.
The proof is completed if we change the order of integration and
differentiation.
\end{proof}

\begin{lemma}\label{SH}
If for some $\varepsilon_0>0$ and $k\in\{1, ... , N-1\}$, it holds
\begin{align*}
|S_k(y)|\leq \varepsilon_0, \quad\forall\, y\in B_{R+r_0}\setminus
B_{R+\frac{r_0}{2}},
\end{align*}
then
\begin{align*}
|S_k(y_r)|\lesssim E^{1-\tau_r}\varepsilon_0^{\tau_r}r^{-(1-\tau_r)},
\end{align*}
where $y_r = P^{(k+1)} +(0,0,-r)$ with $r$ being small enough and $\tau_r=\theta
r^\delta$ with the positive constants $\theta\in(0,1)$ and $\delta$ depending on
$r_0,\kappa, N,A$.
\end{lemma}

\begin{proof}
It follows from Lemma \ref{S_sol} that $S_k$ satisfies $(\Delta +
\kappa^2)S_k(y) = 0$ in $\mathcal{K}_k$. Moreover, by the explicit forms of
$S_k(y)$ and $\Phi_k(x,y)$, we have
\[
|S_k(y)|\leq C E\int_{U_k}\frac{1}{|x-y|^4}\mathrm{d}x\leq
CE\int_{|y-P^{(k+1)}|}^\infty\rho^{-2}\mathrm{d}\rho\leq CE|y-P^{(k+1)}|^{-1},
\]
where $C>0$ is a constant depending on $\kappa,r_0$. By Proposition \ref{estimate}, we have for $r>0$
small enough that 
\begin{align*}
|S_k(y_r)|\lesssim E^{1-\tau_r}\varepsilon_0^{\tau_r}r^{-(1-\tau_r)},
\end{align*}
which completes the proof.
\end{proof}

Multiplying both sides of \eqref{eqn} by $\Phi_k(x, y)$ for $y\in
B_{R+r_0}\setminus B_{R+\frac{r_0}{2}}$ and using integration by parts, we
have
\begin{equation}\label{integral_whole}
\begin{split}
&\int_{\Omega} f(x)\Phi_k(x, y){\rm d}x\\
=&\int_{B_R} f(x)\Phi_k(x, y){\rm d}x\\
=& \int_{B_R}\left[(\Delta+\kappa^2)u(x)\right]\Phi_k(x,y)\mathrm{d}x\\
=&\int_{B_R}u(x)(\Delta_x+\kappa^2)\Phi_k(x,y)\mathrm{d}x+\int_{\partial B_R}
\left[\partial_{\nu(x)}u(x)\Phi_k(x, y) - \partial_{\nu(x)}\Phi_k(x,
y)u(x)\right]{\rm d}s\\
=& \int_{\partial B_R} \left[\partial_{\nu(x)}u(x)\Phi_k(x, y) -
\partial_{\nu(x)}\Phi_k(x, y)u(x)\right]{\rm d}s,
\end{split}
\end{equation}
where $\nu$ is the unit outer normal vector on $\partial B_R$.

First, note that for $k=0$,
\[
S_0(y)=\int_{\Omega} f(x)\Phi_0(x, y){\rm d}x.
\]
Also notice that
\[
\int_{\partial
B_R}|\Phi_0(\cdot,y)|^2+|\partial_\nu\Phi_0(\cdot,y)|^2\mathrm{d}s\leq C
\]
for $y\in B_{R+r_0}\setminus
B_{R+\frac{r_0}{2}}$, where $C$ depends on $R,\kappa,r_0$.
Notice that $u\vert_{\mathbb{R}^3\setminus B_R}$ is the solution to the exterior problem for the Helmholtz equation
\[
\Delta u+\kappa^2 u=0\quad\text{in }\mathbb{R}^3\setminus B_R
\]
along with the radiation condition \eqref{src}.
For the above exterior problem, it is shown in \cite[Theorem 2.6.4]{Nedelec}
that there exists a bounded operator $\mathcal{N}: H^1(\partial B_R)\rightarrow
L^2(\partial B_R)$, which is called exterior Dirichlet-to-Neumann map, such
that 
\[
\partial_\nu u = \mathcal{N} u \quad \text{on} ~\partial B_R.
\]
Hence, the Neumann data $\partial_\nu u$ on $\partial B_R$ can be obtained once the Dirichlet date $u$ is available on $\partial B_R$.
Therefore, we obtain the following estimate
\begin{equation*}
\int_{\partial B_R} (|\partial_\nu u|^2 + \kappa^2|u|^2){\rm d}s = \int_{\partial B_R} (|\mathcal{N} u|^2 + \kappa^2|u|^2){\rm d}s\leq C\|u\|^2_{H^1(\partial B_R)}\leq C \epsilon^2,
\end{equation*}
where $C$ depends on $\kappa$ and $R$.
Therefore by \eqref{integral_whole}, we obtain
\begin{equation}\label{est_S0}
\left\vert S_0(y)\right\vert\lesssim\epsilon, \quad y\in B_{R+r_0}\setminus
B_{R+\frac{r_0}{2}}.
\end{equation}

First we prove a logarithmic-type stability. Denote
$\delta_0=\epsilon$ and $\delta_j = \|f\|_{L^\infty(W_j)}, ~ j = 1,\cdots, N$.
We will inductively prove that the following estimates hold:  
\begin{equation}\label{deltak_est}
\delta_j\leq\omega_j(\epsilon),
\end{equation}
where $\omega_0(\epsilon)\leq \omega_1(\epsilon)\leq\cdots\leq \omega_N(\epsilon)$ for any small $\epsilon>0$ and
\begin{equation*}
\lim_{\epsilon\rightarrow 0}\omega_j(\epsilon)=0
\end{equation*}
 for each $j$. The estimate \eqref{deltak_est} is clearly true for $j=0$, for which
$\omega_0(\epsilon)=\epsilon$, by invoking \eqref{est_S0}. We now assume that the estimate \eqref{deltak_est} is true for $j= k$, and deduce the estimate for
$j=k+1$.
 
Recall that
\[
\begin{split}
S_k(y) = &\int_{U_k}f(x)\Phi_k(x, y){\rm d}x\\
=&\int_{\Omega}f(x)\Phi_k(x, y){\rm d}x-\int_{W_k}f(x)\Phi_k(x, y){\rm d}x.
\end{split}
\]
Thus we have the estimate
\begin{equation}\label{comb1}
|S_{k}(y)|\leq \left\vert\int_{\Omega}f(x)\Phi_k(x,y)\mathrm{d}x\right\vert+\left\vert\int_{W_k}f(x)\Phi_k(x,y)\mathrm{d}x\right\vert.
\end{equation}
Similar to \eqref{est_S0}, we have
\begin{equation}\label{comb2}
\left\vert\int_{\Omega}f(x)\Phi_k(x,y)\mathrm{d}x\right\vert\leq C\epsilon
\end{equation}
for  $y\in B_{R+r_0}\setminus B_{R+\frac{r_0}{2}}$. 
For the estimate of the second term in the right hand side of \eqref{comb1}, first notice that $|x-y|>Cr_0$ for $x\in W_k$ and $y\in B_{R+r_0}\setminus B_{R+\frac{r_0}{2}}$, and therefore
\[
|\Phi_k(x,y)|\leq\frac{C}{|x-y|^4}\leq \frac{C}{r_0^4}.
\]
Also we have $|f(x)|\leq C\omega_k(\epsilon)$ for $x\in W_k$  by the hypothesis for induction. Therefore
\begin{equation}\label{comb3}
\left\vert\int_{W_k}f(x)\Phi_k(x,y)\mathrm{d}x\right\vert\leq
C\omega_k(\epsilon)
\end{equation}
for $y\in B_{R+r_0}\setminus B_{R+\frac{r_0}{2}}$.
Combining the estimates \eqref{comb1}--\eqref{comb3}, we obtain 
\begin{align*}
|S_{k}(y)|\lesssim(\epsilon + \omega_k(\epsilon)),\quad y\in B_{R+r_0}\setminus
B_{R+\frac{r_0}{2}}.
\end{align*}
Note that the above estimate is also valid for $k=0$, for which $W_0=\emptyset$.
Now let $y_r = P^{(k+1)}+(0,0,-r)$, where $r$ is small enough. By Lemma \ref{SH}, we
have,
\begin{equation}\label{estimate_Sk}
|S_{k}(y_r)|\lesssim r^{-1}\omega_k(\epsilon)^{\tau_r}.
\end{equation}

Next, we write
\[
S_{k}(y_r) = I_1 + I_2,
\]
where
\begin{align*}
I_1 &=\int_{B_{r_0}(y_r)\cap D_{k+1}} f(x)\Phi_k(x, y_r){\rm d}x,\\
I_2 &= \int_{U_{k}\backslash (B_{r_0}(y_r)\cap D_{k+1})} f(x)\Phi_k(x, y_r){\rm
d}x.
\end{align*}
\begin{figure}[htbp]
\centering
\includegraphics[width=0.3\textwidth]{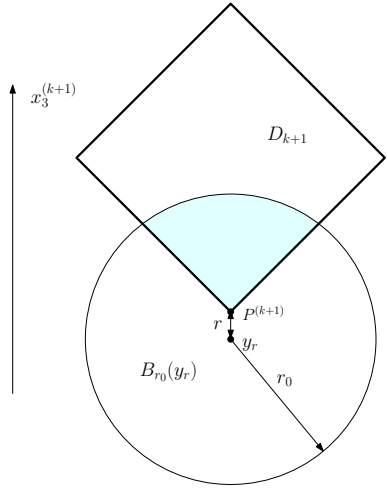}
\caption{The shaded region is $B_{r_0}(y_r)\cap D_{k+1}$.}
\label{nearcorner}
\end{figure}
The region $B_{r_0}(y_r)\cap D_{k+1}$ is depicted in Figure \ref{nearcorner}.
First it is easy to verify that
\begin{align}\label{I_2}
|I_2|\lesssim 1.
\end{align}
Combining \eqref{estimate_Sk} and \eqref{I_2} yields 
\begin{align}\label{estimate_2}
|I_1|\lesssim r^{-1}\omega_k(\epsilon)^{\tau_r}+1. 
\end{align}
Since $f(x)=c_{k+1}$ on $D_{k+1}$, we have
\[
|I_1|=|c_{k+1}|\left\vert\int_{B_{r_0}(y_r)\cap D_{k+1}} \Phi_k(x, y_r){\rm d}x\right\vert.
\]
By \eqref{bp}, we have
\[
\left\vert\int_{B_{r_0}(y_r)\cap D_{k+1}} \Phi_k(x, y_r){\rm d}x\right\vert=\left\vert\int_{\mathcal{C}'(r,r_0,\alpha^{(k+1)})} \Phi(x){\rm d}x\right\vert\geq C_0r^{-1}-C_1r^{-1/2},
\]
where $C_0,C_1$ two positive constants.
Together with \eqref{estimate_2}, we obtain
\[
C_0|c_{k+1}|r^{-1}\lesssim|I_1|+r^{-1/2}\leq r^{-1}\omega_k(\epsilon)^{\tau_r}+r^{-1/2}.
\]
Multiplying above inequality by $r$ gives 
\[
|c_{k+1}|\lesssim \omega_k(\epsilon)^{\tau_r}+r^{1/2},
\]
where $r>0$ is small enough. Taking 
\[
r=|\log\omega_k(\epsilon)|^{-\frac{1}{2\delta}},
\]
we obtain
\[
|c_{k+1}|\lesssim |\log\omega_k(\epsilon)|^{-\frac{1}{4\delta}}.
\]
Hence
\[
\delta_{k+1}\lesssim \omega_{k+1}(\epsilon):=|\log\omega_k(\epsilon)|^{-\frac{1}{4\delta}}.
\]
Remember that $\delta>0$ depends on $r_0,\kappa, N,A$. Then it is easy to verify that $\lim_{\epsilon\rightarrow 0}\omega_{k+1}(\epsilon)=0$. This completes the induction and we can now conclude that
\begin{equation}\label{log_stability}
\|f\|_{L^\infty(\Omega)}\lesssim \omega_N(\epsilon),
\end{equation}
where $\lim_{\epsilon\rightarrow 0}\omega_{N}(\epsilon)=0$.

The final Lipschitz-type stability is an almost immediate consequence of
\eqref{log_stability} since we are recovering a finite number of parameters. To
be rigorous, we use an abstract theorem in \cite[Theorem 2.1 and
Remark 2.2]{Bourgeois}. We also refer to
\cite[Proposition 5]{BV}.

\begin{lemma}
Let $\mathcal{X}$ and $\mathcal{Y}$ be two Banach spaces. Assume that $\mathcal{X}$ is of finite dimension $N$, and $T:\mathcal{X}\rightarrow\mathcal{Y}$ is a linear bounded operator. Let $\mathcal{K}$ be a compact and convex subset of $\mathcal{X}$. If $T$ is injective, then there exists a constant $C_N>0$ such that
\[
\|x\|_{\mathcal{X}}\leq C_N\|Tx\|_{\mathcal{Y}},
\]
for any $x\in\mathcal{K}$.\label{abs_lemma}
\end{lemma}

For our problem, set
\[
\mathcal{X}=\mathbb{C}^N,\quad \mathcal{Y}=H^1(\partial B_R),\quad
\mathcal{K}=\{(c_1,\cdots,c_N)\in\mathbb{C}^N: |c_j|\leq E ~ \forall
j=1,2,\cdots, N\}.
\]
We consider the linear operator $T:\mathbb{C}^N\rightarrow H^1(\partial B_R)$
such that
\[
T(c_1,\cdots,c_N)\mapsto u\vert_{\partial B_R},
\]
where $u$ solves \eqref{eqn} with $f$ being given by the form \eqref{formf}. The
boundedness of $T$ follows directly from \eqref{est_u}, and the injectivity
results from \eqref{log_stability}. Therefore a Lipschitz stability follows
immediately by Lemma \ref{abs_lemma}, which finishes the proof of Theorem
\ref{main_theorem}.


\section{Conclusion}\label{conclusion}

We have presented the Lipschitz stability for the inverse source
scattering problem of the three-dimensional Helmholtz equation in a
homogeneous background medium, where the source is assumed be a piecewise
constant function. The analysis requires the Dirichlet data only. The proof relies
on the construction of singular solutions and the quantitative estimate of
unique continuation of the solutions for elliptic-type equations. A possible
continuation of this work is to study the corresponding stability estimates of
the inverse source problems for elastic and electromagnetic waves, where the
fundamental solutions are tensors and therefore more sophisticated analysis is needed. We
will report the progress elsewhere in the future.

\subsection*{Acknowledgements}

The authors want to express their sincere gratitude to the referees
whose invaluable comments have helped to improve this paper.

\end{document}